\documentclass[a4paper,10pt,reqno]{amsart}

\usepackage{amssymb}
\usepackage{mathrsfs}
\usepackage{enumerate}
\usepackage[hidelinks]{hyperref}
\usepackage{cleveref}
\usepackage{amsrefs}
\usepackage{amsthm}
\usepackage{tikz}

\numberwithin{equation}{section}

\newtheorem{theorem}{Theorem}[section]
\newtheorem{lemma}[theorem]{Lemma}
\newtheorem{proposition}[theorem]{Proposition}
\newtheorem{corollary}[theorem]{Corollary}
\newtheorem{conjecture}[theorem]{Conjecture}
\newtheorem{hypothesis}[theorem]{Hypothesis}

\theoremstyle{definition}
\newtheorem{definition}[theorem]{Definition}

\newcommand{\mc}[1]{\mathcal{#1}}

\newcommand{\supp}[1]{\operatorname{supp}(#1)}

\newcommand{\stack}[1]{\operatorname{stack}(#1)}
\newcommand{\clear}[1]{\operatorname{clear}(#1)}

\newcommand{\conf}[2]{\operatorname{Conf}_{#2}(#1)}
\newcommand{\pmove}[2]{#1\Rightarrow #2}
\newcommand{\pmapsto}[2]{#1\mapsto #2}
\newcommand{\Gmapsto}[3]{#3:\pmapsto{#1}{#2}}
\newcommand{\poper}[3]{#3{[\pmove{#1}{#2}]}}
\newcommand{\norm}[1]{\|#1\|}
\newcommand{\Pebl}[2]{\operatorname{Pebl}_{#2}{(#1)}}

\newcommand{\econf}[2]{c^{#1}_{#2}}
\newcommand{\pull}[1]{\operatorname{p}(#1)}
\newcommand{\push}[1]{\operatorname{s}(#1)}
\newcommand{\Leaf}{\operatorname{Leaf}}
\newcommand{\topp}[2]{#1\ll #2}

\author[T. Csern{\'a}k]{Tam{\'a}s Csern{\'a}k}
\address{University of Pannonia, Hungary}
\email{tamas@csernak.com}

\author[L. Soukup]{Lajos Soukup}
\address{HUN-REN R{\'e}nyi Institute of Mathematics, Budapest, Hungary}
\email{soukup@renyi.hu}

\title[The Almost Stacked Hypothesis]{The Almost Stacked Hypothesis: A
Conjectural Analogue of Sj\"ostrand's Cover Pebbling Theorem}
\date{\today}

\begin{document}

\begin{abstract}
We study two graph pebbling parameters, the stacking number and the clearing
number, through the Almost Stacked Hypothesis (ASH). This hypothesis asserts
that these thresholds can be determined by testing only configurations in
which at most one vertex carries more than one pebble. We prove that every
almost stacked configuration of size $2^{n+1}-1$ on $C_{2n}$ is stackable and
that every almost stacked configuration of size $3\cdot2^n-2$ on $C_{2n+1}$
is clearable. Together with the known lower bounds from~\cite{CsSo26}, these
results show that ASH implies
\[
\stack{C_{2n}}=2^{n+1}-1
\qquad\text{and}\qquad
\clear{C_{2n+1}}=3\cdot2^n-2.
\]
For a finite tree $T$, we introduce an explicit invariant
$\operatorname{estim}(T)$. We prove unconditionally that
$\stack{T}\ge\operatorname{estim}(T)$ and prove the reverse inequality under
ASH\@. Consequently, ASH yields $\stack{T}=\operatorname{estim}(T)$, and we
conjecture that this equality holds unconditionally. Finally, we study
perfectly pebblable graphs: finite connected non-bipartite graphs whose
clearing number has the minimum possible value
\(\clear{G}=|V(G)|+1\). Every complete graph with at least three vertices is
perfectly pebblable, which might suggest that perfect pebblability requires
high edge density. Assuming ASH, however, we show that this is not the case.
We give a sufficient criterion involving strong edge-triangulation and
Hamiltonian-path and path-cover conditions in vertex-deleted subgraphs and use
it to construct two explicit infinite families of perfectly pebblable graphs
with edge density tending to zero, one of which has only a linear number of
edges.
\end{abstract}

\maketitle

\section{Introduction}

Graph pebbling starts with a finite connected graph whose vertices carry
pebbles. A single pebbling step along an edge removes two pebbles from one
endpoint and places one pebble on the other endpoint. In this paper we study
two ways in which such a process can succeed. A configuration is \emph{stacked}
if all remaining pebbles lie on one vertex, and it is \emph{cleared} if only
one pebble remains. The stacking and clearing numbers are the least integers
$t$ such that every configuration of $t$ pebbles is stackable or clearable,
respectively.

The Almost Stacked Hypothesis (ASH) is a proposed reduction principle: to
determine the stacking and clearing thresholds, it should be enough to check
almost stacked configurations, namely configurations in which at most one
vertex carries more than one pebble. We show that ASH gives exact formulas for
$\stack{C_{2n}}$ and $\clear{C_{2n+1}}$. For finite trees $T$, we introduce an
explicit invariant $\operatorname{estim}(T)$, prove unconditionally that
$\stack{T}\ge\operatorname{estim}(T)$, and prove the reverse inequality under
ASH\@. We conjecture that the resulting equality holds unconditionally.

A finite connected non-bipartite graph $G$ is called \emph{perfectly
pebblable} if its clearing number has the minimum possible value
$\clear{G}=|V(G)|+1$. Complete graphs with at least three vertices are
perfectly pebblable, which might suggest that this property requires high edge
density. Assuming ASH, however, we prove otherwise: we give a sufficient
criterion for perfect pebblability and use it to construct two sparse infinite
families of perfectly pebblable graphs, one of which has only a linear number
of edges.

\subsection{Background and relation to earlier work}

The classical pebbling number $\pi(G)$, introduced by Chung~\cite{Cu89}, is
the least integer $m$ such that, from every configuration of $m$ pebbles on
$G$, one can move a pebble to any prescribed vertex. Early foundational work
includes~\cites{Mo92,PaSn95}; for a modern overview of graph pebbling,
see~\cite{HuKe21}. The subject has developed in several directions, including
cover pebbling~\cites{CrCuFeHuPuSzTu05,Sj05} and Hurlbert's general framework
for target configurations~\cite{Hu13}. Related variants and graph classes
include cover rubbling~\cite{HaKe20} and pebbling in powers of
paths~\cite{AlHu23}.

We introduced the stacking and clearing numbers in~\cite{CsSo26}. As in
Hurlbert's framework, these parameters specify a class of admissible final
configurations. Their distinguishing feature is that the final vertex is not
prescribed: a successful stacking sequence may finish at any vertex, while a
successful clearing sequence may leave its last pebble anywhere. The Almost
Stacked Hypothesis studied here is directly analogous to Sj\"ostrand's cover
pebbling theorem~\cite{Sj05}, according to which the extremal initial
configuration for a positive demand function may be taken to be stacked.
Because the target vertex is not fixed in our setting, isolated pebbles away
from the main pile can create genuine obstructions; this leads naturally to
almost stacked, rather than fully stacked, extremal configurations.

Our interest in these questions began with Hurlbert's seminar
talk~\cite{Hurlbert2023}. The computational data associated with our earlier
work are available in~\cite{So26}; in particular, they verify the stacking
assertion of ASH for all connected graphs with at most seven vertices and its
clearing assertion for all non-bipartite such graphs.

\section{Basic Definitions}

Let $G$ be a finite simple graph.
A \emph{configuration} on $G$ is a function $c\colon V(G)\to\mathbb{N}$, where
$c(v)$ denotes the number of pebbles at $v$. Its \emph{size} is
\begin{equation*}
\norm{c}=\sum_{v\in V(G)}c(v),
\end{equation*}
and its \emph{support} is
\begin{equation*}
\supp{c}=\{v\in V(G):c(v)>0\}.
\end{equation*}
Let $\conf{G}{}$ denote the set of all configurations on $G$, and let
\begin{equation*}
\conf{G}{t}=\{c\in\conf{G}{}:\norm{c}=t\}.
\end{equation*}
For configurations $c,d\in\conf{G}{}$, define $c\oplus d$ by
\begin{equation*}
(c\oplus d)(v)=c(v)+d(v)\qquad \text{for all } v\in V(G).
\end{equation*}
For a nonnegative integer $m$, the configuration $m\cdot c$ is defined in the
natural way. If $d(v)\le c(v)$ for every $v\in V(G)$, then $c\ominus d$ denotes
the configuration defined by
\begin{equation*}
(c\ominus d)(v)=c(v)-d(v)\qquad\text{for all }v\in V(G).
\end{equation*}

For $v\in V(G)$, let $e_v=e_v^G$ denote the \emph{unit configuration} at $v$,
that is, the configuration defined by
\begin{equation*}
e_v(v)=1
\qquad\text{and}\qquad
e_v(w)=0 \ \text{for } w\in V(G)\setminus\{v\}.
\end{equation*}

If $c\in\conf{G}{}$ and $uv\in E(G)$ with $c(u)\ge2$, then the pebbling step
$\pmove{u}{v}$ is \emph{applicable} to $c$. Its result is the configuration
$c'=\poper{u}{v}{c}$ defined by
\begin{equation*}
c'(u)=c(u)-2,\qquad
c'(v)=c(v)+1,\qquad
c'(w)=c(w)\ \text{for } w\notin\{u,v\}.
\end{equation*}
For configurations $c,c'\in\conf{G}{}$ we write
\begin{equation*}
\pmapsto{c}{c'}
\end{equation*}
to mean that $c'$ is reachable from $c$ by a sequence of pebbling steps. When
the graph needs to be indicated explicitly, we write $\Gmapsto{c}{c'}{G}$. We
also define
\begin{equation*}
\Pebl{c}{G}=\{c'\in\conf{G}{}:\pmapsto{c}{c'}\}.
\end{equation*}
A configuration is called \emph{terminal} if no pebbling step is applicable
to it, equivalently, if it has at most one pebble at each vertex.

\begin{definition}[Stacking and clearing]
A configuration is \emph{stacked} if its support has size $1$, and it is
\emph{stacked at $v$} if $\supp{c}=\{v\}$. A configuration is \emph{cleared}
if its size is $1$. A configuration $c$ is \emph{stackable} if some stacked
configuration lies in $\Pebl{c}{G}$; it is \emph{stackable at $v$} if some
configuration stacked at $v$ lies in $\Pebl{c}{G}$. It is \emph{clearable} if
some cleared configuration lies in $\Pebl{c}{G}$.

Following~\cite{CsSo26}, we define the stacking number only for connected
graphs and the clearing number only for connected non-bipartite graphs. If
$G$ is connected, then the \emph{stacking number} $\stack{G}$ is the least
integer $t\ge2$ such that every configuration in $\conf{G}{t}$ is stackable.
If $G$ is connected and non-bipartite, then the \emph{clearing number}
$\clear{G}$ is the least integer $t\ge2$ such that every configuration in
$\conf{G}{t}$ is clearable.
\end{definition}

\section{The Almost Stacked Hypothesis}

\begin{definition}[Almost stacked configurations]
Let $G$ be a graph. A configuration $c$ on $G$ is called
\emph{almost stacked} if there exists a vertex $v_0\in V(G)$ such that
\begin{equation*}
c(u)\le 1 \qquad \text{for all } u\in V(G)\setminus\{v_0\}.
\end{equation*}
We then say that $c$ is \emph{almost stacked at $v_0$}.
\end{definition}

\begin{hypothesis}[Almost Stacked Hypothesis]\label{HY_almoststacked}
Let $G$ be a finite connected graph.
\begin{enumerate}[(1)]
\item $\stack{G}$ is the least integer $t\ge2$ such that every almost stacked
configuration in $\conf{G}{t}$ is stackable.
\item If $G$ is non-bipartite, then $\clear{G}$ is the least integer $t\ge2$
such that every almost stacked configuration in $\conf{G}{t}$ is clearable.
\end{enumerate}
\end{hypothesis}

\subsection{Intuition}

The Almost Stacked Hypothesis is motivated by the heuristic that extremal
counterexamples should have at most one vertex carrying more than one pebble.
A pair of pebbles away from a distinguished vertex permits an immediate move
and hence provides local flexibility, suggesting that such a pair should not
make a configuration harder to stack or clear. The expected obstructions are
the isolated single pebbles distributed through the graph: they record the
residues that must be absorbed, pushed forward, or cleared. Thus ASH asserts
that, when determining the stacking and clearing thresholds, it suffices to
consider configurations consisting of one distinguished vertex, possibly
carrying a large pile, together with a $0$--$1$ configuration on the remaining
vertices.

This point of view is directly analogous to Sj\"ostrand's cover pebbling
theorem, where the worst configurations may be taken to be stacked. In the
present setting the target vertex is not fixed, and single pebbles away from
the main pile can affect the parity and divisibility constraints of a pebbling
sequence. Thus the natural analogue of a stacked extremal configuration is not
a fully stacked configuration, but an almost stacked one: one distinguished
vertex that may carry a large pile, and at most one pebble on every other
vertex.

For cycles this intuition becomes especially concrete. Once the distinguished
vertex is fixed as $v_1$, the two directions around the cycle determine two
path-like arms starting at $v_1$. The $0$--$1$ patterns on these arms determine
how much fuel is needed to pull isolated pebbles back to the distinguished
vertex or to push a pebble forward while clearing the intervening vertices.
This is exactly what the Pullback and Pushforward Lemmas measure.

\section{Cycle Applications}

\subsection{Background results}

The following background results are taken from~\cite{CsSo26}; their precise
locations are indicated in the headings.

\begin{theorem}[Complete graphs, {\cite{CsSo26}*{Theorem~8.1}}]\label{tm:complete-graph}
If $n\ge 3$, then
\begin{equation*}
\clear{K_n}=\stack{K_n}=n+1.
\end{equation*}
\end{theorem}

\begin{definition}[Extremal path configurations, {\cite{CsSo26}*{Definition~7.2}}]\label{df:path-conf}
Let $n\ge2$ and $m\in\mathbb{N}$, and let $P_n$ have vertices
$v_1,\dots,v_n$ in natural order. We denote by $\econf{m}{n}$ the configuration
on $P_n$ with $m$ pebbles on $v_1$, one pebble on $v_n$, and no pebbles on the
remaining vertices.
Equivalently,
\begin{equation*}
\econf{m}{n}=m\cdot e_{v_1}\oplus e_{v_n}.
\end{equation*}
\end{definition}

\begin{theorem}[Path lower-bound configuration, {\cite{CsSo26}*{Theorem~7.3}}]\label{tm:path-lower}
Let $n\ge2$, and let $m<2^n$ with $m\equiv 2^n\pmod 3$.
Then the configuration $\econf{m}{n}$ is not stackable.
\end{theorem}

\begin{theorem}[Stacking number of paths, {\cite{CsSo26}*{Theorem~7.4}}]\label{tm:pn}
For every $n\ge 2$,
\begin{equation*}
\stack{P_n}=2^n-1.
\end{equation*}
\end{theorem}

\begin{corollary}[Even-cycle lower bound, {\cite{CsSo26}*{Corollary~7.5}}]\label{cr:stack_c_2n}
For every $n\ge2$,
\begin{equation*}
\stack{C_{2n}}\ge 2^{n+1}-1.
\end{equation*}
\end{corollary}

\begin{corollary}[Odd-cycle clearing lower bound, {\cite{CsSo26}*{Corollary~6.3}}]\label{cr:clear-odd-cycle-lower}
For every $n\ge1$,
\begin{equation*}
\clear{C_{2n+1}}\ge 3\cdot 2^n-2.
\end{equation*}
\end{corollary}

\begin{lemma}[Monotonicity above the order]\label{lm:threshold-monotonicity}
Let $G$ be a finite graph and let $t>|V(G)|$. If every configuration of size
$t$ is stackable (respectively, clearable), then the same is true at every size
larger than $t$.
\end{lemma}

\begin{proof}
Let $c$ be a configuration of size greater than $t$. As long as its size is
greater than $t$, it has more than $|V(G)|$ pebbles and therefore admits a
pebbling move. Since each move decreases the size by one, repeated moves
produce a configuration of size $t$, to which the assumed conclusion applies.
\end{proof}

\subsection{Pullback and Pushforward Lemmas}

We first record an elementary path operation. Let $P_k$ have vertices
$v_1,\dots,v_k$ in natural order, let $1\le h<k$, and let $d$ be a $0$--$1$
configuration on the internal vertices of the path. Then
\begin{equation}\label{eq:elementary-path-operation}
\Gmapsto{2^h\cdot e_{v_1}\oplus d}{e_{v_{h+1}}\oplus d}{P_k}.
\end{equation}
Indeed, for $i=1,\dots,h$, perform $2^{h-i}$ moves
$\pmove{v_i}{v_{i+1}}$, using at each stage only the pebbles produced by the
preceding moves. This leaves every pebble of $d$ in place. We shall use this
operation without further comment.

In the following lemmas, $d:X\to\{0,1\}$ denotes a $0$--$1$ configuration on
$X$.

\begin{lemma}[Pullback Lemma]\label{lm:pullback}
Assume that $k\ge2$ and $d:\{v_2,\dots,v_k\}\to\{0,1\}$. Put
\begin{equation*}
J=\{i:2\le i\le k,\ d(v_i)=1\}.
\end{equation*}
Define
\begin{equation}\label{eq:pullback-cost}
\pull{d}=
\begin{cases}
1, & \text{if }J=\varnothing,\\[2mm]
\displaystyle
2^{j-1}+\sum_{\substack{2\le i<j\\ i\notin J}}2^{i-1},
& \text{if }J\ne\varnothing,\quad j=\max J.
\end{cases}
\end{equation}
Then
\begin{equation}\label{eq:pullback-reach}
\Gmapsto{\pull{d}\cdot e_{v_1}\oplus d}{e_{v_1}}{P_k}.
\end{equation}
Moreover, writing $q=|\supp{d}|$, we have
\begin{equation}\label{eq:pullback-bound}
\pull{d}\le
\begin{cases}
1, & q=0,\\
2^k-2q, & q>0.
\end{cases}
\end{equation}
\end{lemma}

\begin{proof}
If $J=\varnothing$, then~\eqref{eq:pullback-reach} is immediate. Suppose that
$J\ne\varnothing$ and let $j=\max J$. For each $i<j$ with $i\notin J$, use
$2^{i-1}$ pebbles at $v_1$ to place one pebble at $v_i$. Then use
$2^{j-1}$ pebbles at $v_1$ to place one additional pebble at $v_j$.
At this point every vertex $v_2,\dots,v_{j-1}$ contains one pebble, $v_j$
contains two pebbles, and every vertex beyond $v_j$ is empty. The moves
\begin{equation*}
\pmove{v_i}{v_{i-1}}
\qquad (i=j,j-1,\dots,2)
\end{equation*}
now produce $e_{v_1}$.

For the estimate, let $q=|J|$. From~\eqref{eq:pullback-cost},
\begin{align*}
\pull{d}
&=2^j-2-\sum_{\substack{2\le i<j\\ i\in J}}2^{i-1}\\
&\le 2^j-2-2(q-1)
=2^j-2q
\le2^k-2q.
\end{align*}
\end{proof}

\begin{lemma}[Pushforward Lemma]\label{lm:pushforward}
Assume that $k\ge2$ and $d:\{v_2,\dots,v_{k-1}\}\to\{0,1\}$. Define
\begin{equation}\label{eq:pushforward-cost}
\push{d}
=2+\sum_{\substack{2\le i<k\\ d(v_i)=0}}2^{i-1}
=2^{k-1}-\sum_{i=2}^{k-1}2^{i-1}d(v_i).
\end{equation}
Then
\begin{equation}\label{eq:pushforward-reach}
\Gmapsto{\push{d}\cdot e_{v_1}\oplus d}{e_{v_k}}{P_k}.
\end{equation}
Moreover,
\begin{equation}\label{eq:pushforward-bound}
\push{d}\le2^{k-1}-2|\supp{d}|.
\end{equation}
\end{lemma}

\begin{proof}
For each internal vertex $v_i$ with $d(v_i)=0$, use $2^{i-1}$ pebbles at
$v_1$ to place one pebble at $v_i$. By~\eqref{eq:pushforward-cost}, two pebbles
remain at $v_1$, and every internal vertex contains one pebble. The moves
\begin{equation*}
\pmove{v_i}{v_{i+1}}
\qquad (i=1,2,\dots,k-1)
\end{equation*}
produce $e_{v_k}$. Finally, every occupied internal vertex subtracts a term
$2^{i-1}\ge2$ from the second expression in~\eqref{eq:pushforward-cost}, which
gives~\eqref{eq:pushforward-bound}.
\end{proof}

We recall the weight notation from~\cite{CsSo26}. Given a finite graph $H$, a
function $R\colon V(H)\to\mathbb{R}_{>0}$, and a configuration $f$ on $H$, the
\emph{$R$-weight} of $f$ is
\begin{equation*}
w_R(f)=\sum_{v\in V(H)}R(v)f(v).
\end{equation*}
For later use on $P_k$, define $R(v_i)=2^{i-1}$. A move
$\pmove{v_i}{v_{i+1}}$ preserves the $R$-weight, while a move
$\pmove{v_{i+1}}{v_i}$ decreases it by $3\cdot2^{i-1}$. Hence every path move
preserves the $R$-weight modulo $3$. Consequently, the two lemmas also give
\begin{equation}\label{eq:path-cost-congruences}
\pull{d}+\sum_{i=2}^k2^{i-1}d(v_i)\equiv1\pmod3
\end{equation}
and
\begin{equation}\label{eq:path-push-congruence}
\push{d}+\sum_{i=2}^{k-1}2^{i-1}d(v_i)
\equiv2^{k-1}\pmod3.
\end{equation}

We shall also use the following elementary observation. If a vertex $r$ has a
neighbor $x$, then the three moves
\begin{equation}\label{eq:three-move-reduction}
\pmove{r}{x},\qquad \pmove{r}{x},\qquad \pmove{x}{r}
\end{equation}
reduce the number of pebbles at $r$ by three and restore the number at $x$.
Thus a multiple of three surplus pebbles at $r$ can be discarded while a
fixed positive base pile is retained. More explicitly, a stack of size
$3q+1$ can be reduced to one pebble at $r$, while a stack of size $3q+2$ can
be reduced to two pebbles at $r$ and then cleared by one move
$\pmove{r}{x}$. Hence, whenever the stacking vertex has a neighbor, every
stacked configuration whose size is not divisible by three is clearable.

\subsection{Even cycles}

\begin{theorem}[Almost stacked configurations on even cycles]\label{tm:ash-even-cycle-upper}
For every $n\ge2$, an almost stacked configuration $c$ on $C_{2n}$ is
stackable provided
\begin{equation*}
\norm{c}\ge2^{n+1}-1.
\end{equation*}
\end{theorem}

\begin{proof}
Let $G=C_{2n}$, with vertices $v_1,\dots,v_{2n}$ in cyclic order. By symmetry,
we may assume that $c$ is almost stacked at $v_1$. Put
\begin{equation*}
U=\{v_1,v_2,\dots,v_{n+1}\},
\qquad
W=\{v_1,v_{2n},v_{2n-1},\dots,v_{n+1}\}.
\end{equation*}
For $H\in\{U,W\}$, let
\begin{equation*}
H'=H\setminus\{v_{n+1}\},
\qquad
H''=H\setminus\{v_1,v_{n+1}\}.
\end{equation*}
Thus $G[H']\cong P_n$ and $G[H]\cong P_{n+1}$.

We distinguish two cases.

\medskip
\noindent\textbf{Case 1: $c(v_{n+1})=0$.}
For $H\in\{U,W\}$, let
\begin{equation*}
d_H=c\restriction H'',
\qquad q_H=|\supp{d_H}|,
\qquad p_H=\pull{d_H},
\end{equation*}
where the Pullback Lemma is applied to the path $G[H']$.

If $q_U+q_W=0$, then $c$ is already stacked at $v_1$. Suppose that
$q_U+q_W>0$. If both $q_U$ and $q_W$ are positive, then
\begin{equation*}
p_U+p_W
\le2^{n+1}-2(q_U+q_W).
\end{equation*}
If exactly one of them is positive, then the pullback cost on the empty arm is
one, and hence
\begin{equation*}
p_U+p_W
\le2^n-2(q_U+q_W)+1
\le2^{n+1}-2(q_U+q_W).
\end{equation*}
On the other hand,
\begin{equation*}
c(v_1)=\norm{c}-q_U-q_W
\ge2^{n+1}-1-q_U-q_W.
\end{equation*}
Since $q_U+q_W\ge1$, it follows that $c(v_1)\ge p_U+p_W$.
Apply the two pullback sequences. They clear every vertex different from
$v_1$, so $c$ is stackable.

\medskip
\noindent\textbf{Case 2: $c(v_{n+1})=1$.}
Let
\begin{equation*}
d=c\restriction U'',
\qquad q=|\supp{d}|,
\qquad s=\push{d}.
\end{equation*}
The Pushforward Lemma, applied to $G[U]\cong P_{n+1}$, gives
\begin{equation}\label{eq:even-push}
s\le2^n-2q
\quad\text{and}\quad
\Gmapsto{s\cdot e_{v_1}\oplus d}{e_{v_{n+1}}}{G[U]}.
\end{equation}
There are at most $2n-1$ pebbles outside $v_1$, and therefore
\begin{equation*}
c(v_1)\ge2^{n+1}-2n\ge2^n\ge s,
\end{equation*}
where $2^n\ge2n$ for $n\ge2$. Thus the pushforward sequence is applicable.
It clears $U''$ and places a second pebble at $v_{n+1}$. Now perform
$\pmove{v_{n+1}}{v_{n+2}}$, and call the resulting configuration $c^*$.

The pushforward sequence has $s+q-1$ moves, and the final move across
$v_{n+1}v_{n+2}$ contributes one more. Hence
\begin{align*}
\norm{c^*}
&=\norm{c}-s-q\\
&\ge2^{n+1}-1-(2^n-2q)-q\\
&\ge2^n-1.
\end{align*}
Moreover, $\supp{c^*}\subseteq W'$. Since $G[W']\cong P_n$,
and $2^n-1>n$, Theorem~\ref{tm:pn} and
Lemma~\ref{lm:threshold-monotonicity} show that $c^*$ is stackable. Therefore
$c$ is stackable.
\end{proof}

\subsection{Odd cycles}

\begin{theorem}[Almost stacked configurations on odd cycles]\label{tm:clearoddcircuit}
If $n\ge1$, then an almost stacked configuration $c$ on $C_{2n+1}$ is
clearable provided
\begin{equation*}
\norm{c}\ge3\cdot2^n-2.
\end{equation*}
\end{theorem}

\begin{proof}
If $n=1$, then $C_3\cong K_3$. Since
$\clear{K_3}=4>|V(K_3)|$, Theorem~\ref{tm:complete-graph} and
Lemma~\ref{lm:threshold-monotonicity} give the result. Assume that $n\ge2$,
and put
\begin{equation*}
A=2^n.
\end{equation*}

Let $G=C_{2n+1}$ have vertices $v_1,\dots,v_{2n+1}$ in cyclic order. By
symmetry, $c$ is almost stacked at $r=v_1$, and, after reversing the cyclic
order if necessary, we may assume that
\begin{equation*}
\alpha=c(v_{n+1})\le c(v_{n+2})=\beta.
\end{equation*}
Thus $(\alpha,\beta)$ is one of $(0,0),(0,1),(1,1)$. Put
\begin{equation*}
a=v_{n+1},\qquad b=v_{n+2},
\end{equation*}
and let
\begin{equation*}
U=\{v_1,v_2,\dots,v_{n+1}\},
\qquad
W=\{v_1,v_{2n+1},v_{2n},\dots,v_{n+2}\}.
\end{equation*}
The two induced paths $G[U]$ and $G[W]$ have length $n$, and
$e=ab$ is the only edge not contained in them. Let
\begin{equation*}
U^\circ=U\setminus\{r,a\},
\qquad
W^\circ=W\setminus\{r,b\},
\end{equation*}
and define
\begin{equation*}
q_H=|\supp{c\restriction H^\circ}|,
\qquad
D_H=\sum_{x\in H^\circ}2^{\operatorname{dist}_G(r,x)}c(x)
\qquad(H\in\{U,W\}).
\end{equation*}
Finally, set $M=c(r)$. Since every off-root pebble is a singleton,
\begin{equation}\label{eq:odd-root-lower}
M\ge3A-2-q_U-q_W-\alpha-\beta.
\end{equation}

Define $R(x)=2^{\operatorname{dist}_G(r,x)}$ and use the associated $R$-weight.
Every move in $G-e$ preserves $w_R$ modulo $3$: a move away from $r$ preserves
the weight exactly, while a move toward $r$ decreases it by three times the
weight of the vertex nearer to $r$.

For $H\in\{U,W\}$, apply the Pushforward Lemma to the binary configuration on
$H^\circ$, and denote its cost by $s_H$. The exact
formula~\eqref{eq:pushforward-cost} gives
\begin{equation}\label{eq:odd-push-cost}
s_H=A-D_H\le A-2q_H,
\end{equation}
and the weight congruence gives
\begin{equation}\label{eq:odd-push-weight}
s_H+D_H\equiv A\pmod3.
\end{equation}

\medskip
\noindent\textbf{Case 1: $\alpha=0$ and $w_R(c)\not\equiv0\pmod3$.}
Let $U^-=U\setminus\{a\}$. Pull back the configuration on $U^\circ$ along
$G[U^-]\cong P_n$, and pull back the configuration on
$W\setminus\{r\}$ along $G[W]\cong P_{n+1}$. Denote the two costs by $p_U$
and $p_W$, respectively, and put
\begin{equation*}
x=q_U,\qquad y=q_W+\beta.
\end{equation*}
We claim that $p_U+p_W\le M$. If $x=y=0$, then
$p_U+p_W=2\le3A-2\le M$. If $x=0<y$, then
\begin{equation*}
p_U+p_W\le1+2A-2y\le3A-2-y\le M.
\end{equation*}
If $y=0<x$, then
\begin{equation*}
p_U+p_W\le A-2x+1\le3A-2-x\le M.
\end{equation*}
Finally, if $x,y>0$, then
\begin{equation*}
p_U+p_W\le3A-2(x+y)\le3A-2-(x+y)\le M.
\end{equation*}
Here we used $A\ge4$ and~\eqref{eq:odd-root-lower}.

The two pullback sequences clear every off-root pebble and produce a
configuration stacked at $r$. No move used $e$, so the size of the final stack
is congruent to $w_R(c)$ modulo $3$ and is therefore not divisible by three.
The observation following the path lemmas shows that this stack is clearable.

\medskip
\noindent\textbf{Case 2: $\alpha=0$ and $w_R(c)\equiv0\pmod3$.}
Put
\begin{equation*}
\ell=M-s_U-s_W.
\end{equation*}
By~\eqref{eq:odd-root-lower} and~\eqref{eq:odd-push-cost},
\begin{equation}\label{eq:odd-ell-lower}
\ell
\ge A-2-\beta+(D_U-q_U)+(D_W-q_W)
\ge A-2-\beta+q_U+q_W.
\end{equation}
In particular, $\ell\ge0$. Moreover,
\begin{equation}\label{eq:odd-ell-congruence}
0\equiv w_R(c)
\equiv(s_U+D_U)+(s_W+D_W)+\ell+\beta A
\equiv(2+\beta)A+\ell\pmod3.
\end{equation}

Suppose first that $\beta=1$. Then $\ell\equiv0\pmod3$.
Use~\eqref{eq:three-move-reduction} to discard the $\ell$ surplus pebbles at
$r$,
leaving $s_U+s_W$ pebbles there. Apply the two pushforward sequences. The
resulting configuration has one pebble at $a$ and two at $b$, and is clearable
by the moves $\pmove{b}{a}$ and $\pmove{a}{b}$.

Now suppose that $\beta=0$. Equation~\eqref{eq:odd-ell-congruence} gives
$\ell\equiv A\pmod3$. The lower bound~\eqref{eq:odd-ell-lower} and this
congruence imply that $\ell\ge A$: this is immediate when
$q_U+q_W\ge2$, while for $q_U+q_W\le1$ it follows because the largest integer
below $A$ congruent to $A$ modulo $3$ is $A-3$. Thus
\begin{equation*}
\ell^*=\ell-A
\end{equation*}
is a nonnegative multiple of three. Discard this surplus, apply both
pushforward sequences, and use the remaining $A$ pebbles at $r$ to send one
additional pebble to $b$ along $G[W]$, whose internal vertices are now empty.
Again the resulting configuration has one pebble at $a$ and two at $b$, and
hence is clearable.

\medskip
\noindent\textbf{Case 3: $\alpha=\beta=1$ and
$w_R(c)\not\equiv A\pmod3$.}
First apply the pushforward sequence on $W^\circ$. It clears $W^\circ$ and
places a second pebble at $b$. Perform $\pmove{b}{a}$, so that $b$ becomes
empty and $a$ contains two pebbles.

Starting at $a$, move this pair toward $r$ through the maximal occupied suffix
of $U^\circ$. More explicitly, let $j\in\{1,\dots,n\}$ be minimal such that
$c(v_i)=1$ whenever $j<i\le n$, and perform
\begin{equation*}
\pmove{v_{i+1}}{v_i}
\qquad(i=n,n-1,\dots,j).
\end{equation*}
The resulting configuration on $U^\circ$ is again a $0$--$1$ configuration;
call it $f$, and put $q_f=|\supp f|$. Pull $f$ back to $r$ along
$U^-=U\setminus\{a\}$, and denote the cost by $p_f$.

There are enough pebbles at $r$ for these operations. If $f\ne0$, then
$q_f\ge1$ and the initial configuration on $U^\circ$ was not identically one,
so $q_U\le n-2$. Hence
\begin{align*}
M-(s_W+p_f)
&\ge A-4-q_U+q_W+2q_f\\
&\ge A-n\ge0.
\end{align*}
If $f=0$, then $q_U=n-1$ and $p_f=1$, so
\begin{equation*}
M-(s_W+p_f)\ge2A-n-4+q_W\ge0.
\end{equation*}

These operations clear every off-root vertex and therefore produce a
configuration stacked at $r$. Exactly one move used $e$, namely
$\pmove{b}{a}$, and that move decreases the $R$-weight by $A$. Thus the size
of the final stack is congruent to
\begin{equation*}
w_R(c)-A\not\equiv0\pmod3.
\end{equation*}
The final stack is clearable.

\medskip
\noindent\textbf{Case 4: $\alpha=\beta=1$ and
$w_R(c)\equiv A\pmod3$.}
If one of the configurations on $U^\circ,W^\circ$ is empty, relabel the two
arms so that it is the configuration on $U^\circ$. If both are nonempty,
choose either arm as $U$.

Let $p$ be the pullback cost for $c\restriction U^\circ$ along
$U^-=U\setminus\{a\}$, and let $s=s_W$ be the pushforward cost on the other
arm. Put
\begin{equation*}
C=(p-1)+s+A,
\qquad
L=M-C.
\end{equation*}
We first show that $L\ge0$. If $q_U=0$, then $p=1$ and
\begin{equation*}
C=s+A\le2A-2q_W\le M+1.
\end{equation*}
If $q_U,q_W>0$, then
\begin{equation*}
C\le3A-1-2(q_U+q_W)\le M+1.
\end{equation*}
Thus $L\ge-1$ in either case.

The path congruences give
\begin{equation*}
p+D_U\equiv1\pmod3,
\qquad
s+D_W\equiv A\pmod3.
\end{equation*}
Since
\begin{equation*}
w_R(c)=M+D_U+D_W+2A\equiv A\pmod3,
\end{equation*}
we obtain
\begin{align*}
L
&=M-p+1-s-A\\
&\equiv M+D_U+D_W-2A\\
&\equiv w_R(c)-4A
\equiv-3A
\equiv0\pmod3.
\end{align*}
Since $L\ge-1$, it follows that $L\ge0$.

Discard the $L$ surplus pebbles at $r$
using~\eqref{eq:three-move-reduction}, leaving $C$ pebbles there. Apply the
pullback sequence on $U^\circ$. It clears this arm and leaves
\begin{equation*}
C-p+1=s+A
\end{equation*}
pebbles at $r$. Next apply the pushforward sequence on $W^\circ$, leaving
$A$ pebbles at $r$ and placing a second pebble at $b$. Use the remaining
$A$ pebbles to send one more pebble to $b$ along $G[W]$, whose internal
vertices are now empty. We have reached
\begin{equation*}
e_a\oplus3\cdot e_b.
\end{equation*}
This configuration is clearable by
\begin{equation*}
\pmove{b}{a},\qquad
\pmove{a}{b},\qquad
\pmove{b}{a}.
\end{equation*}
This completes all cases.
\end{proof}

\subsection{Consequences of ASH}

\begin{theorem}[ASH for even cycles]\label{tm:ash-even-cycle-equality}
Assume ASH for even cycles. Then, for every $n\ge2$,
\begin{equation*}
\stack{C_{2n}}=2^{n+1}-1.
\end{equation*}
\end{theorem}

\begin{proof}
By Theorem~\ref{tm:ash-even-cycle-upper}, ASH implies
$\stack{C_{2n}}\le2^{n+1}-1$. Corollary~\ref{cr:stack_c_2n} gives the reverse
inequality.
\end{proof}

\begin{theorem}[ASH for odd cycles]\label{tm:ash-odd-cycle-equality}
Assume ASH for odd cycles. Then, for every $n\ge1$,
\begin{equation*}
\clear{C_{2n+1}}=3\cdot2^n-2.
\end{equation*}
\end{theorem}

\begin{proof}
By Theorem~\ref{tm:clearoddcircuit}, ASH implies
$\clear{C_{2n+1}}\le3\cdot2^n-2$. Corollary~\ref{cr:clear-odd-cycle-lower}
gives the reverse inequality.
\end{proof}

\section{Tree Applications}

\begin{definition}[The tree estimate]
Let $T$ be a finite tree and let $r\in V(T)$. Define
$\operatorname{leaf}_T(r)$ to be the number of leaves of $T$ different from
$r$. All degrees and distances below are taken in $T$.
Define also
\begin{equation*}
{\rho}_T(r)=
\sum\{\deg(v)\cdot 2^{\operatorname{dist}_T(r,v)}:
v\in V(T),\ v=r \text{ or } \deg(v)>1\}.
\end{equation*}
Finally, let
\begin{equation*}
\operatorname{estim}(T)=
\max\left\{2,\;
1+\max\{{\rho}_T(r)+\operatorname{leaf}_T(r):r\in V(T)\}\right\}.
\end{equation*}

If $c$ is a configuration on $T$ and $r\in V(T)$, let $T(c,r)$ denote the
minimal subtree of $T$ containing $r$ and $\supp{c}$.
\end{definition}

\begin{theorem}[Tree lower bound]\label{tm:tree}
If $T$ is a finite tree, then $\stack{T}\ge \operatorname{estim}(T)$.
\end{theorem}

\begin{proof}
If $|V(T)|=1$, then $\stack{T}=2=\operatorname{estim}(T)$. We may therefore
assume that $|V(T)|\ge2$. In this case the second term in the definition of
$\operatorname{estim}(T)$ is larger than two: for every root $r$ we have
$\rho_T(r)\ge\deg_T(r)\ge1$, and $T$ has a leaf different from $r$. Thus we
may choose a root
$r\in V(T)$ such that
\begin{equation*}
\operatorname{estim}(T)-1={\rho}_T(r)+\operatorname{leaf}_T(r).
\end{equation*}
Let $\Leaf(T)$ be the set of leaves of $T$ different from $r$. Root $T$ at
$r$, and write $\topp{u}{v}$ if $uv\in E(T)$ and $u$ is the predecessor of
$v$ in this rooted tree.

Let $c$ be the configuration on $T$ defined by
\begin{equation}\label{eq:c}
c(u)=
\begin{cases}
1, & \text{if } u\in \Leaf(T),\\
{\rho}_T(r), & \text{if } u=r,\\
0, & \text{otherwise.}
\end{cases}
\end{equation}
Then
\begin{equation}\label{eq:norm}
\norm{c}={\rho}_T(r)+|\Leaf(T)|=\operatorname{estim}(T)-1.
\end{equation}
We will show that $c$ is not stackable.

Define
\begin{equation*}
R(v)=2^{\operatorname{dist}_T(r,v)}
\qquad (v\in V(T)).
\end{equation*}
Given a configuration $d$ and an applicable pebbling step $\pmove{x}{y}$,
define its $R$-loss by
\begin{equation*}
\ell_R(\pmove{x}{y})=2R(x)-R(y).
\end{equation*}
Thus
\begin{equation*}
w_R\bigl(\poper{x}{y}{d}\bigr)=w_R(d)-\ell_R(\pmove{x}{y}).
\end{equation*}
For an edge $\topp{u}{v}$ we have $R(v)=2R(u)$, and hence
\begin{equation}\label{eq:tree-loss}
\ell_R(\pmove{u}{v})=0
\qquad\text{and}\qquad
\ell_R(\pmove{v}{u})=3R(u).
\end{equation}

The initial $R$-weight is
\begin{align}
w_R(c)
&={\rho}_T(r)+\sum_{v\in\Leaf(T)} R(v) \notag\\
&=\sum\{R(u)+R(v): \topp{u}{v}\} \notag\\
&=\sum\{3R(u): \topp{u}{v}\}. \label{eq:tree-weight}
\end{align}
For the middle equality, observe that the root is counted once for each child,
an internal non-root vertex is counted once as a child and once for each of
its children, and a non-root leaf is supplied by the final sum in the first
line.
In particular, $w_R(c)\equiv0\pmod 3$. Since every loss in
equation~\eqref{eq:tree-loss} is divisible by $3$, the $R$-weight is preserved
modulo $3$ during every pebbling sequence.

Suppose, toward a contradiction, that $c$ is stackable. Let
$c'\in\Pebl{c}{T}$ be stacked at $s$, and choose a pebbling sequence
\begin{equation*}
c=c_0,c_1,\dots,c_N=c'
\end{equation*}
with $c_{i+1}=\poper{x_i}{y_i}{c_i}$ for $i<N$.
Summing the losses along the sequence gives
\begin{equation}\label{eq:loss-sum}
w_R(c')+\sum_{i<N}\ell_R(\pmove{x_i}{y_i})=w_R(c).
\end{equation}

Since $c'$ is stacked at $s$,
\begin{equation*}
w_R(c')=c'(s)R(s).
\end{equation*}
The congruence $w_R(c')\equiv w_R(c)\equiv0\pmod 3$ and the fact that
$R(s)$ is a power of $2$ imply that $3$ divides $c'(s)$. As $c'(s)>0$, we have
\begin{equation}\label{eq:stack-weight-lower}
w_R(c')\ge 3R(s).
\end{equation}

Let
\begin{equation*}
r=w_0,w_1,\dots,w_m=s
\end{equation*}
be the path from $r$ to $s$, and let
\begin{equation*}
\mc F=\{\{w_i,w_{i+1}\}:0\le i<m\}.
\end{equation*}
We claim that for every edge $\topp{u}{v}$ with $\{u,v\}\notin\mc F$, the
sequence contains a move $\pmove{v}{u}$. Indeed, let $T_v$ be the rooted
subtree consisting of $v$ and all its descendants. Since $\{u,v\}\notin\mc F$,
the final stacking vertex $s$ does not belong to $T_v$. The subtree $T_v$
contains a leaf of $T$ different from $r$, so it contains a pebble in the
initial configuration, but no pebble in the final configuration. Consequently,
the step that first empties $T_v$ must move a pebble across the edge $uv$ from
$v$ to $u$. This proves the claim.

By~\eqref{eq:tree-loss}, the claim gives
\begin{equation}\label{eq:off-path-loss}
\sum_{i<N}\ell_R(\pmove{x_i}{y_i})
\ge \sum\{3R(u): \topp{u}{v},\ \{u,v\}\notin\mc F\}.
\end{equation}
Combining equations~\eqref{eq:loss-sum},~\eqref{eq:tree-weight},
and~\eqref{eq:off-path-loss}, we obtain
\begin{align*}
w_R(c')
&\le \sum\{3R(u): \topp{u}{v},\ \{u,v\}\in\mc F\}\\
&=\sum_{i=0}^{m-1}3\cdot 2^i\\
&=3(2^m-1).
\end{align*}
On the other hand, $R(s)=2^m$, so~\eqref{eq:stack-weight-lower} gives
\begin{equation*}
w_R(c')\ge 3\cdot 2^m,
\end{equation*}
a contradiction.

Thus $c$ is a non-stackable configuration of size
$\operatorname{estim}(T)-1$. To pass from this obstruction to a threshold
bound, note first that $\stack{T}\ge |V(T)|+1$: for every
$2\le t\le |V(T)|$, placing one pebble on each of $t$ distinct vertices gives
a terminal non-stacked configuration. If, contrary to the conclusion, we had
$\stack{T}<\operatorname{estim}(T)$, then
$\stack{T}>|V(T)|$. By definition, every configuration of size $\stack{T}$ is
stackable, and Lemma~\ref{lm:threshold-monotonicity} extends this conclusion to
every larger size. Since
$\operatorname{estim}(T)-1\ge\stack{T}$, every configuration of size
$\operatorname{estim}(T)-1$ would therefore be stackable. This contradicts the
configuration $c$ constructed above. Hence
$\stack{T}\ge \operatorname{estim}(T)$.
\end{proof}

The preceding theorem gives the lower bound unconditionally. We do not know how
to prove the matching upper bound for arbitrary configurations. The next result
proves the required upper bound for almost stacked configurations, and so gives
the matching upper bound under the Almost Stacked Hypothesis.

\begin{theorem}[Almost stacked configurations on trees]\label{tm:tree-almost-stacked}
Let $T$ be a finite tree, let $r\in V(T)$, and let $c$ be a configuration on
$T$ that is almost stacked at $r$. If, for $S=T(c,r)$,
\begin{equation*}
\norm{c}\ge {\rho}_S(r)+\operatorname{leaf}_S(r)+1,
\end{equation*}
then $c$ is stackable at $r$.
\end{theorem}

\begin{proof}
We argue by induction on $|V(S)|$. If $S$ consists only of $r$, then the
hypothesis gives $\norm{c}\ge1$. Since $\supp{c}\subseteq\{r\}$, the
configuration $c$ is stacked at $r$.

Root $S$ at $r$, and choose a vertex $s$ of maximum distance from $r$ among
the vertices that have at least one child. Every non-root leaf of $S$ belongs
to $\supp{c}$, since otherwise it could be deleted from the minimal subtree
$T(c,r)$. Since $c$ is almost stacked at $r$, every such leaf carries exactly
one pebble. By the choice of $s$, all children of $s$ are leaves and therefore
carry one pebble.

Suppose first that $s=r$. Then $S$ is a star with center $r$. Put
$t=\deg_S(r)$. The hypothesis gives
\begin{equation*}
c(r)+t=\norm{c}\ge {\rho}_S(r)+\operatorname{leaf}_S(r)+1=2t+1,
\end{equation*}
so $c(r)\ge t+1$. For each leaf $x$, perform
\begin{equation*}
\pmove{r}{x},\qquad \pmove{x}{r}.
\end{equation*}
Each pair of moves clears $x$ and decreases the number of pebbles at $r$ by
one. After all leaves have been cleared, at least one pebble remains at $r$.

We may therefore assume that $s\ne r$. Put
\begin{equation*}
k=\operatorname{dist}_S(r,s),
\qquad
n=\deg_S(s).
\end{equation*}
Thus $s$ has $n-1$ children, all of which are leaves.

By the elementary path operation~\eqref{eq:elementary-path-operation}, if a
vertex $u$ is at distance $h$ from $r$, then $2^h$ pebbles at $r$ can place one
additional pebble at $u$ by moving only away from $r$ along the $r$--$u$ path
and leaving all pre-existing pebbles on that path in place.

Let
\begin{equation*}
q=|\supp{c}\setminus\{r\}|,
\qquad
L=\operatorname{leaf}_S(r).
\end{equation*}
Every non-root leaf of $S$ is occupied, so $q-L$ is the number of occupied
internal vertices different from $r$. The term of ${\rho}_S(r)$ corresponding
to $s$ is $n\cdot2^k$, and every other occupied internal vertex contributes at
least one. Consequently, if $c(s)=0$, then
\begin{equation*}
{\rho}_S(r)\ge n\cdot2^k+q-L,
\end{equation*}
whereas if $c(s)=1$, then
\begin{equation*}
{\rho}_S(r)\ge n\cdot2^k+q-L-1.
\end{equation*}
Since
\begin{equation*}
c(r)=\norm{c}-q
\ge {\rho}_S(r)+L+1-q,
\end{equation*}
the preceding estimates give
\begin{equation}\label{eq:tree-root-resources}
c(r)\ge
\begin{cases}
n\cdot2^k+1, & c(s)=0,\\
n\cdot2^k, & c(s)=1.
\end{cases}
\end{equation}
Thus there are enough pebbles at $r$ for the operations described below: they
use $n\cdot2^k$ root pebbles when $c(s)=0$, and
$(n-1)2^k\le n2^k$ root pebbles when $c(s)=1$.

If $c(s)=0$, choose one child $x$ of $s$. Use $2^{k+1}$ pebbles at $r$ to
place one additional pebble at $x$, and then perform $\pmove{x}{s}$. This
clears $x$ and leaves one pebble at $s$. For each remaining child $y$ of $s$,
use $2^k$ pebbles at $r$ to place one additional pebble at $s$, and then
perform
\begin{equation*}
\pmove{s}{y},\qquad \pmove{y}{s}.
\end{equation*}
This clears $y$ and again leaves one pebble at $s$. If $c(s)=1$, apply this
latter procedure to every child of $s$.

Let $c'$ be the resulting configuration and let $S'=T(c',r)$. Then $c'$ is
almost stacked at $r$, and $S'$ is obtained from $S$ by deleting the $n-1$
children of $s$. Hence
\begin{equation*}
{\rho}_S(r)-{\rho}_{S'}(r)=n\cdot2^k
\qquad\text{and}\qquad
\operatorname{leaf}_S(r)-\operatorname{leaf}_{S'}(r)=n-2.
\end{equation*}

Moving one pebble to distance $h$ by the elementary path operation uses
$2^h-1$ pebbling steps. Therefore, if $c(s)=0$, the number of steps used above
is
\begin{equation*}
2^{k+1}+(n-2)(2^k+1)=n\cdot2^k+n-2.
\end{equation*}
If $c(s)=1$, the number is
\begin{equation*}
(n-1)(2^k+1)\le n\cdot2^k+n-2,
\end{equation*}
because $k\ge1$. Since every pebbling step decreases the size of the
configuration by one, it follows that
\begin{align*}
\norm{c'}
&\ge \norm{c}-n\cdot2^k-n+2\\
&\ge {\rho}_{S'}(r)+\operatorname{leaf}_{S'}(r)+1.
\end{align*}
Since $|V(S')|<|V(S)|$, the induction hypothesis shows that $c'$ is stackable
at $r$. Hence $c$ is stackable at $r$.
\end{proof}

\begin{theorem}[Tree consequence of ASH]\label{tm:tree-ash}
If ASH holds, then every finite tree $T$ satisfies
\begin{equation*}
\stack{T}\le \operatorname{estim}(T).
\end{equation*}
\end{theorem}

\begin{proof}
By ASH, to prove the upper bound for $\stack{T}$ it is enough to consider
almost stacked configurations. Let $c$ be almost stacked at $r$, suppose that
\(\norm{c}=\operatorname{estim}(T)\), and put $S=T(c,r)$.

For every subtree $S$ of $T$ containing $r$,
\begin{equation}\label{eq:tree-estimate-monotone}
{\rho}_S(r)+\operatorname{leaf}_S(r)
\le {\rho}_T(r)+\operatorname{leaf}_T(r).
\end{equation}
Indeed, one can obtain $T$ from $S$ by adjoining leaves one at a time, and
each such addition can only increase the left-hand expression. More
explicitly, suppose that a new leaf $x$ is attached to a vertex $u$. If
$u=r$, then the expression increases by two. If $u\ne r$ is a leaf before the
addition, then $u$ ceases to be a leaf, $x$ replaces it as a leaf, and the new
term $2\cdot2^{\operatorname{dist}_T(r,u)}$ is added to $\rho$. If $u$ is
already an internal vertex, then $\rho$ increases by
$2^{\operatorname{dist}_T(r,u)}$ and the leaf count increases by one. Therefore
\begin{align*}
\norm{c}
&=\operatorname{estim}(T)\\
&\ge {\rho}_T(r)+\operatorname{leaf}_T(r)+1\\
&\ge {\rho}_S(r)+\operatorname{leaf}_S(r)+1.
\end{align*}
Theorem~\ref{tm:tree-almost-stacked} shows that $c$ is stackable. Hence ASH
gives $\stack{T}\le\operatorname{estim}(T)$.
\end{proof}

\begin{corollary}[Tree formula under ASH]\label{cr:tree-formula-ash}
Assume ASH\@. Then every finite tree $T$ satisfies
\begin{equation*}
\stack{T}=\operatorname{estim}(T).
\end{equation*}
\end{corollary}

\begin{proof}
Combine Theorems~\ref{tm:tree} and~\ref{tm:tree-ash}.
\end{proof}

\begin{conjecture}\label{conj:tree-equality}
If $T$ is a finite connected tree, then
\begin{equation*}
\stack{T}=\operatorname{estim}(T).
\end{equation*}
\end{conjecture}

\section{Perfectly Pebblable Graphs}\label{sc:perfectly-pebblable}

\subsection{Basic properties}

The clearing number of a connected non-bipartite graph cannot be smaller than
the number of vertices plus one. Indeed, if \(n=|V(G)|\), then, for each
integer \(t\) with \(2\le t\le n\), the configuration that places one pebble
on each of \(t\) distinct vertices is terminal and is not clearable. This
motivates the following definition.

\begin{definition}\label{df:perfectly-pebblable}
A finite connected non-bipartite graph \(G\) is \emph{perfectly pebblable} if
\begin{equation*}
\clear{G}=|V(G)|+1.
\end{equation*}
\end{definition}

Thus, perfectly pebblable graphs are precisely the graphs whose clearing
number is as small as possible. By Theorem~\ref{tm:complete-graph}, every
complete graph with at least three vertices is perfectly pebblable.

Perfect pebblability is preserved by the addition of edges without changing
the vertex set. Indeed, if \(H\) is a spanning subgraph of \(G\) and \(H\) is
perfectly pebblable, then every pebbling move available in \(H\) is also
available in \(G\). Thus
\(\clear{G}\le\clear{H}=|V(G)|+1\), while the reverse inequality follows from
the preceding lower bound.

Perfect pebblability also imposes the following Hamiltonicity condition.

\begin{proposition}\label{pr:perfect-Hamiltonian}
If \(G\) is perfectly pebblable, then, for every \(u\in V(G)\), the graph \(G\)
has a Hamiltonian path starting at \(u\).
\end{proposition}

\begin{proof}
Fix \(u\in V(G)\), and let \(c\) place two pebbles at \(u\) and one pebble at
every other vertex. Since \(\norm{c}=|V(G)|+1\), the configuration \(c\) is
clearable.

Consider a clearing sequence starting from \(c\), and call the successive
locations of the unique pair of pebbles the \emph{visited vertices}.
Initially, \(u\) is the only vertex from which a pebbling move can be made.
More generally, an induction shows that, after the sequence has visited
distinct vertices
\begin{equation*}
u_1=u,u_2,\dots,u_t,
\end{equation*}
the current configuration has two pebbles at \(u_t\), no pebbles at
\(u_1,\dots,u_{t-1}\), and one pebble at every unvisited vertex. If
\(t<|V(G)|\) and the next move sends a pebble to a previously visited vertex,
the resulting configuration is terminal but not clearable. Hence a clearing
sequence must send the next pebble to an unvisited neighbor \(u_{t+1}\). It
follows inductively that the sequence visits every vertex before it clears.
Once every vertex has been visited, two pebbles remain at the last vertex,
and one final move clears the configuration. Therefore
\begin{equation*}
u_1u_2\cdots u_{|V(G)|}
\end{equation*}
is a Hamiltonian path starting at \(u\).
\end{proof}

We now give a condition under which the stacking and clearing numbers coincide.

\begin{definition}\label{df:edge-triangulated}
Call a graph \emph{edge-triangulated} if every edge belongs to a triangle.
Call a graph \emph{strongly edge-triangulated} if every pair of distinct
vertices has a common neighbor. In particular, every strongly
edge-triangulated graph is edge-triangulated. Notice that the strong condition
is imposed on all pairs of vertices, not only on adjacent pairs.
\end{definition}

We shall also use the elementary bound
\begin{equation}\label{eq:stack-order-lower}
\stack{G}\ge |V(G)|+1.
\end{equation}
Indeed, for each integer \(t\) with \(2\le t\le |V(G)|\), a configuration
that places one pebble on each of \(t\) distinct vertices is terminal and is
not stacked. Thus no integer at most \(|V(G)|\) has the defining universal
property of the stacking number.

\begin{lemma}\label{lm:edge-triangulated}
If \(G\) is a finite connected edge-triangulated graph with at least two
vertices, then
\begin{equation*}
\clear{G}=\stack{G}.
\end{equation*}
\end{lemma}

\begin{proof}
Every clearable configuration is stackable, so
\(\stack{G}\le\clear{G}\). For the reverse inequality, let \(c\) be a
configuration of size \(\stack{G}\). Then \(c\) is stackable.

Starting from any stack of at least four pebbles, the three moves
\begin{equation*}
\pmove{u}{v},\qquad
\pmove{u}{v},\qquad
\pmove{v}{u},
\end{equation*}
where \(uv\in E(G)\), reduce the size of the stack by three. Repeating this
reduction while at least four pebbles remain leaves a stack of size one, two,
or three. In the second case, one further move produces a stack of size one.
Hence \(c\) reaches a stacked configuration \(d\) with
\(\norm{d}\in\{1,3\}\).

If \(\norm{d}=1\), then \(c\) is clearable. Suppose that \(\norm{d}=3\), with
the three pebbles at \(v\). Since
\begin{equation*}
\norm{c}=\stack{G}\ge |V(G)|+1\ge4,
\end{equation*}
the sequence from \(c\) to \(d\) contains at least one move. Let
\(\pmove{u}{v}\) be its final move. Immediately before this move there are two
pebbles at each of \(u\) and \(v\), and no pebbles elsewhere: the move
produces \(3e_v\) and decreases the size by one, so the preceding
configuration must be \(2e_u\oplus2e_v\). Because the edge
\(uv\) belongs to a triangle, there is a common neighbor \(w\) of \(u\) and
\(v\). The moves
\begin{equation*}
\pmove{u}{w},\qquad
\pmove{v}{w},\qquad
\pmove{w}{u}
\end{equation*}
produce a cleared configuration. Thus every configuration of size
\(\stack{G}\) is clearable, and therefore
\(\clear{G}\le\stack{G}\).
\end{proof}

\begin{corollary}\label{cr:perfect-stacking}
Let \(G\) be a finite connected edge-triangulated graph with at least two
vertices. Then \(G\) is perfectly pebblable if and only if
\begin{equation*}
\stack{G}=|V(G)|+1.
\end{equation*}
\end{corollary}

\begin{proof}
Since \(G\) has an edge and every edge of \(G\) belongs to a triangle, \(G\)
is non-bipartite. The assertion now follows directly from
Lemma~\ref{lm:edge-triangulated} and
Definition~\ref{df:perfectly-pebblable}.
\end{proof}

\subsection{An ASH criterion}

The preceding result connects perfect pebblability directly with the Almost
Stacked Hypothesis.

\begin{corollary}\label{cr:perfect-ash}
Assume ASH, and let \(G\) be a finite connected edge-triangulated graph with at
least two vertices. If every almost stacked configuration on \(G\) of size
\(|V(G)|+1\) is stackable, then \(G\) is perfectly pebblable.
\end{corollary}

\begin{proof}
By ASH,
\begin{equation*}
\stack{G}\le |V(G)|+1.
\end{equation*}
The lower bound~\eqref{eq:stack-order-lower} gives
\(\stack{G}\ge |V(G)|+1\). Hence
\(\stack{G}=|V(G)|+1\), and
Corollary~\ref{cr:perfect-stacking} completes the proof.
\end{proof}

\begin{definition}\label{df:pathcover}
A \emph{path cover} of a graph $G$ is a family of pairwise vertex-disjoint
paths whose vertex sets cover $V(G)$. Paths consisting of a single vertex are
allowed. The \emph{path-cover number} $\theta(G)$ is the minimum size of a path
cover of $G$; we set $\theta(G)=0$ when $V(G)=\varnothing$.
\end{definition}

\begin{lemma}[Path-clearing lemma]\label{lm:path-clearing}
Let $G$ be strongly edge-triangulated, let
$\mc P=\{P_1,\dots,P_m\}$ be a family of pairwise vertex-disjoint paths in
$G$, and let
\begin{equation*}
v_0\in V(G)\setminus\bigcup_{i=1}^m V(P_i).
\end{equation*}
Suppose that $d$ is a configuration such that
\begin{equation*}
d(v)=
\begin{cases}
1, & v\in\displaystyle\bigcup_{i=1}^m V(P_i),\\
0, & v\notin\displaystyle\bigcup_{i=1}^m V(P_i)\cup\{v_0\},
\end{cases}
\qquad\text{and}\qquad
d(v_0)\ge5m+1.
\end{equation*}
Then $d$ is stackable at $v_0$.
\end{lemma}

\begin{proof}
If $m=0$, then $d$ is already stacked at $v_0$. Assume that $m\ge1$.
We show how to clear one path using five pebbles from $v_0$. Fix
$P\in\mc P$, and write its vertices in path order as
\begin{equation*}
u_1,u_2,\dots,u_k.
\end{equation*}
Since $G$ is strongly edge-triangulated, there is a common neighbor $a$ of
$v_0$ and $u_1$, and there is a common neighbor $b$ of $v_0$ and $u_k$.
Perform the moves
\begin{equation*}
\pmove{v_0}{a},\qquad
\pmove{v_0}{a},\qquad
\pmove{a}{u_1},
\end{equation*}
followed by
\begin{equation*}
\pmove{u_i}{u_{i+1}}
\qquad (1\le i<k).
\end{equation*}
At this point the path is empty except for two pebbles at $u_k$. Now perform
\begin{equation*}
\pmove{v_0}{b},\qquad
\pmove{u_k}{b},\qquad
\pmove{b}{v_0}.
\end{equation*}
Since the graph is simple, each common neighbor is distinct from the two
vertices for which it was chosen. The net effect of the displayed moves is
\begin{equation*}
d(v_0)\longmapsto d(v_0)-5,
\qquad
d(u_i)\longmapsto0\quad(1\le i\le k),
\end{equation*}
while the number of pebbles at every vertex outside
$V(P)\cup\{v_0\}$ is unchanged. Possible overlaps cause no difficulty: if
$a$ belongs to the path, its original pebble has been restored before the
moves along the path begin, while if $b$ belongs to the path, it is empty
when the last three moves begin. This also covers the possibility $a=b$.
If an auxiliary vertex belongs to another path in $\mc P$, its original
single pebble is restored, so that path is unaffected.

Before clearing the $i$th path, at least
\begin{equation*}
5(m-i+1)+1\ge6
\end{equation*}
pebbles remain at $v_0$, so all the displayed moves are applicable. Indeed,
the first two moves from $v_0$ use four of these pebbles, and the later move
$\pmove{v_0}{b}$ uses two more; the final move restores one pebble
to $v_0$. Repeating the procedure for $P_1,\dots,P_m$ clears every vertex
different from $v_0$ and leaves at least one pebble at $v_0$.
\end{proof}

\begin{theorem}\label{tm:ASH2PP}
Assume ASH, and let $G$ be a finite connected graph with at least two vertices.
Suppose that the following conditions hold:
\begin{enumerate}[(1)]
\item $G$ is strongly edge-triangulated.
\item If $X\subseteq V(G)$ and $|X|\le4$, then
$G[V(G)\setminus X]$ has a Hamiltonian path starting at each of its vertices.
\item If $X\subseteq V(G)$ and $|X|\ge5$, then
\begin{equation*}
\theta\bigl(G[V(G)\setminus X]\bigr)
\le \left\lfloor\frac{|X|}{5}\right\rfloor.
\end{equation*}
\end{enumerate}
Then $G$ is perfectly pebblable.
\end{theorem}

\begin{proof}
Let $n=|V(G)|$, and let $c\in\conf{G}{n+1}$ be almost stacked at $v_0$. Since
every vertex different from $v_0$ carries at most one pebble, we have
$c(v_0)\ge2$. Put
\begin{equation*}
Z=\{v\in V(G):c(v)=0\}.
\end{equation*}
Then $v_0\notin Z$, and
\begin{equation}\label{eq:perfect-main-pile}
n+1=\norm{c}
=c(v_0)+n-|Z|-1,
\end{equation}
so
\begin{equation*}
c(v_0)=|Z|+2.
\end{equation*}
We prove that $c$ is stackable.

\smallskip
\noindent\textbf{Case 1: $c(v_0)=2$.}
In this case $Z=\varnothing$. By condition~(2), $G$ has a Hamiltonian path
starting at $v_0$. Moving the pair of pebbles successively along this path
produces a stacked configuration.

\smallskip
\noindent\textbf{Case 2: $c(v_0)=3$.}
Now $|Z|=1$. If $v_0$ is the only occupied vertex, there is nothing to prove.
Otherwise, condition~(2) gives a Hamiltonian path in $G[V(G)\setminus Z]$
starting at $v_0$. Let $u$ be the second vertex of this path. After the move
$\pmove{v_0}{u}$, the resulting configuration has two pebbles at $u$ and one
pebble at every other vertex of $V(G)\setminus Z$. By condition~(2), this
induced graph has a Hamiltonian path starting at $u$, along which the
configuration can be stacked.

\smallskip
\noindent\textbf{Case 3: $c(v_0)=4$.}
Here $|Z|=2$. If $v_0$ is the only occupied vertex, the configuration is
already stacked. If there is exactly one occupied vertex $u\ne v_0$, then
condition~(2) implies that $u$ and $v_0$ are adjacent. The two moves
\begin{equation*}
\pmove{v_0}{u},\qquad \pmove{u}{v_0}
\end{equation*}
produce a configuration stacked at $v_0$; on the two occupied vertices their
net effect is
\begin{equation*}
4e_{v_0}\oplus e_u\longmapsto3e_{v_0}.
\end{equation*}

It remains to consider the case in which at least two vertices different from
$v_0$ are occupied. By condition~(2), take as $u$ the second vertex on a
Hamiltonian path in $G[V(G)\setminus Z]$ starting at $v_0$; then $u$ is an
occupied neighbor of $v_0$. Applying condition~(2) to $Z\cup\{u\}$, take as
$w$ the second vertex on a Hamiltonian path starting at $v_0$ in the
remaining induced graph; such a second vertex exists because another occupied
vertex besides $v_0$ remains. Thus $w$ is an occupied neighbor
\begin{equation*}
w\in V(G)\setminus\bigl(Z\cup\{u,v_0\}\bigr)
\end{equation*}
of $v_0$. Perform
\begin{equation*}
\pmove{v_0}{u},\qquad
\pmove{u}{v_0},\qquad
\pmove{v_0}{w}.
\end{equation*}
The resulting configuration has two pebbles at $w$ and one pebble at every
other vertex of
\begin{equation*}
V(G)\setminus\bigl(Z\cup\{u\}\bigr).
\end{equation*}
Condition~(2) supplies a Hamiltonian path in
$G[V(G)\setminus(Z\cup\{u\})]$ starting at $w$, so the configuration is
stackable.

\smallskip
\noindent\textbf{Case 4: $c(v_0)=5$.}
Now $|Z|=3$. Again there is nothing to prove if $v_0$ is the only occupied
vertex. Otherwise, let $u$ be the second vertex on a Hamiltonian path in
$G[V(G)\setminus Z]$ starting at $v_0$. Then $u$ is an occupied neighbor of
$v_0$. Perform
\begin{equation*}
\pmove{v_0}{u},\qquad
\pmove{v_0}{u},\qquad
\pmove{u}{v_0}.
\end{equation*}
On $v_0$ and $u$, the net effect is
\begin{equation*}
5e_{v_0}\oplus e_u\longmapsto2e_{v_0}\oplus e_u.
\end{equation*}
The resulting configuration has two pebbles at $v_0$ and one pebble at every
other vertex of $V(G)\setminus Z$. By condition~(2), a Hamiltonian path in
$G[V(G)\setminus Z]$ starting at $v_0$ stacks this configuration.

\smallskip
\noindent\textbf{Case 5: $c(v_0)\ge6$.}
By~\eqref{eq:perfect-main-pile}, $|Z|\ge4$. Apply condition~(3) to
\begin{equation*}
X=Z\cup\{v_0\},
\end{equation*}
whose size is $|Z|+1\ge5$. There is a path cover
\begin{equation*}
\mc P=\{P_1,\dots,P_m\}
\end{equation*}
of $G[V(G)\setminus X]$ such that
\begin{equation*}
m\le\left\lfloor\frac{|Z|+1}{5}\right\rfloor.
\end{equation*}
Consequently,
\begin{equation*}
c(v_0)=|Z|+2\ge5m+1.
\end{equation*}
Lemma~\ref{lm:path-clearing}, applied to this path cover, shows that $c$ is
stackable at $v_0$.

We have proved that every almost stacked configuration in
$\conf{G}{n+1}$ is stackable. Therefore ASH gives
\begin{equation*}
\stack{G}\le n+1.
\end{equation*}
The lower bound~\eqref{eq:stack-order-lower} gives $\stack{G}\ge n+1$.
Hence $\stack{G}=n+1$. Finally, condition~(1)
implies that $G$ is edge-triangulated, and
Lemma~\ref{lm:edge-triangulated} yields
\begin{equation*}
\clear{G}=\stack{G}=n+1.
\end{equation*}
Thus $G$ is perfectly pebblable.
\end{proof}

\subsection{Sparse constructions}

We conclude with two families of perfectly pebblable graphs whose edge
densities tend to zero.

\begin{definition}\label{df:sparse-perfect}
Fix an integer $k\ge2$, and let
\begin{equation*}
\{B_i:i\in\mathbb Z_{k^2}\}
\end{equation*}
be a family of pairwise disjoint sets, each of size five. For
$i,j\in\mathbb Z_{k^2}$, let
\begin{equation*}
\delta(i,j)=\min\{|i-j|,\,k^2-|i-j|\}
\end{equation*}
denote their cyclic distance, where the representatives of $i$ and $j$ are
taken from $\{0,\dots,k^2-1\}$.

Define $G_k$ to be the graph with
\begin{equation*}
V(G_k)=\bigcup_{i\in\mathbb Z_{k^2}}B_i
\end{equation*}
in which each $B_i$ induces a complete graph and, for distinct
$i,j\in\mathbb Z_{k^2}$, every vertex of $B_i$ is adjacent to every vertex of
$B_j$ if and only if
\begin{equation*}
1\le\delta(i,j)\le k-1
\qquad\text{or}\qquad
i\equiv j\pmod{k}.
\end{equation*}
\end{definition}

\begin{theorem}\label{tm:sparse-perfect}
Assume ASH\@. Then $G_k$ is perfectly pebblable for every integer $k\ge2$.
Moreover,
\begin{equation*}
|V(G_k)|=5k^2
\qquad\text{and}\qquad
|E(G_k)|=\frac{75}{2}k^3-\frac{55}{2}k^2.
\end{equation*}
In particular,
\begin{equation*}
|E(G_k)|=O(k^3)=O\bigl(|V(G_k)|^{3/2}\bigr).
\end{equation*}
\end{theorem}

\begin{proof}
Fix $k\ge2$. We verify the three conditions of
Theorem~\ref{tm:ASH2PP}. The graph $G_k$ is finite and connected since
consecutive blocks are completely joined.

First, let $x\in B_i$ and $y\in B_j$ be distinct vertices. Choose
$r\in\{0,\dots,k-1\}$ such that
\begin{equation*}
r\equiv j-i\pmod{k},
\end{equation*}
and put $\ell=i+r$ in $\mathbb Z_{k^2}$. Then
\begin{equation*}
\delta(i,\ell)=r\le k-1
\qquad\text{and}\qquad
\ell\equiv j\pmod{k}.
\end{equation*}
Here $\delta(i,\ell)=r$ because $0\le r\le k-1<k^2/2$.
Choose
\begin{equation*}
z\in B_\ell\setminus\{x,y\},
\end{equation*}
which is possible because $|B_\ell|=5$. If $r=0$, then $x$ and $z$ lie in
the same complete block; otherwise, $1\le\delta(i,\ell)\le k-1$, so they are
adjacent. Also, $y$ and $z$ either lie in the same complete block or lie in
blocks whose indices are congruent modulo $k$. Thus $z$ is adjacent to both
$x$ and $y$. Hence $G_k$ is strongly edge-triangulated, and condition~(1) of
Theorem~\ref{tm:ASH2PP} holds.

Next, let $X\subseteq V(G_k)$ with $|X|\le4$. Since every block has five
vertices,
\begin{equation*}
B_i\setminus X\ne\varnothing
\qquad\text{for every }i\in\mathbb Z_{k^2}.
\end{equation*}
List the vertices of each $B_i\setminus X$ in an arbitrary order, and
concatenate these lists in the cyclic order
\begin{equation*}
B_0,B_1,\dots,B_{k^2-1}.
\end{equation*}
Each block induces a complete graph, and consecutive blocks, including
$B_{k^2-1}$ and $B_0$, are completely joined. The resulting cyclic list is
therefore a Hamiltonian cycle in $G_k-X$. Breaking this cycle immediately
before any prescribed vertex gives a Hamiltonian path starting at that
vertex. Thus condition~(2) holds.

Finally, let $X\subseteq V(G_k)$ with $|X|\ge5$, and set
\begin{equation*}
I=\{i\in\mathbb Z_{k^2}:B_i\subseteq X\}.
\end{equation*}
If $I=\varnothing$, the preceding cyclic construction gives a Hamiltonian
cycle in $G_k-X$. Hence
\begin{equation*}
\theta(G_k-X)=1\le\left\lfloor\frac{|X|}{5}\right\rfloor.
\end{equation*}

Suppose that $I\ne\varnothing$. The set
$\mathbb Z_{k^2}\setminus I$ is the union of at most $|I|$ maximal cyclic
intervals of consecutive indices. For each such interval, list all vertices
of $B_i\setminus X$ block by block in cyclic order. This gives one path since
each nonempty block is complete and consecutive blocks are completely joined.
The resulting paths are pairwise vertex-disjoint and cover every vertex of
$G_k-X$. If every block is contained in $X$, this is the empty path cover.
Therefore
\begin{equation*}
\theta(G_k-X)\le |I|.
\end{equation*}
The blocks indexed by $I$ are pairwise disjoint and each contains five
vertices of $X$, so
\begin{equation*}
5|I|\le |X|.
\end{equation*}
Consequently,
\begin{equation*}
\theta(G_k-X)
\le |I|
\le\left\lfloor\frac{|X|}{5}\right\rfloor,
\end{equation*}
which proves condition~(3). Theorem~\ref{tm:ASH2PP} now implies that $G_k$ is
perfectly pebblable.

It remains to count the vertices and edges. There are $k^2$ blocks of size
five, so
\begin{equation*}
|V(G_k)|=5k^2.
\end{equation*}
The edges inside the blocks contribute
\begin{equation*}
k^2\binom{5}{2}=10k^2
\end{equation*}
edges. Each index has $2(k-1)$ other indices $j$ with
\begin{equation*}
1\le\delta(i,j)\le k-1,
\end{equation*}
so there are $k^2(k-1)$ unordered pairs $\{i,j\}$ of this kind. Each such
pair contributes $25$ edges between its blocks. There are also
$k$ congruence classes modulo $k$, each containing $k$ indices, so the
congruence condition contributes
\begin{equation*}
k\binom{k}{2}
\end{equation*}
further pairs of blocks. These two families of block pairs are disjoint:
distinct indices congruent modulo $k$ have cyclic distance at least $k$.
Hence
\begin{align*}
|E(G_k)|
&=10k^2+25k^2(k-1)+25k\binom{k}{2}\\
&=\frac{75}{2}k^3-\frac{55}{2}k^2.
\end{align*}
The asserted asymptotic estimate follows immediately. More explicitly, the
edge density satisfies
\begin{equation*}
\frac{|E(G_k)|}{\binom{|V(G_k)|}{2}}
=\frac{75k^3-55k^2}{5k^2(5k^2-1)}
\sim\frac{3}{k},
\end{equation*}
and hence tends to zero.
\end{proof}

\begin{definition}\label{df:wheel-perfect}
Fix an integer $k\ge3$. Let
\begin{equation*}
B_*\quad\text{and}\quad \{B_i:i\in\mathbb Z_k\}
\end{equation*}
be pairwise disjoint sets of size five. Define $W_k$ to be the graph with
\begin{equation*}
V(W_k)=B_*\cup\bigcup_{i\in\mathbb Z_k}B_i
\end{equation*}
such that each of the sets $B_*,B_0,\dots,B_{k-1}$ induces a complete graph,
$B_*$ is completely joined to every $B_i$, and $B_i$ is completely joined to
$B_{i+1}$ for every $i\in\mathbb Z_k$. There are no other edges.
\end{definition}

Figure~\ref{fig:wheel-perfect} illustrates this block construction.

\begin{figure}[ht]
\centering
\begingroup
\newcommand{\wheelblock}[3]{%
  \begin{scope}[shift={(#1,#2)}]
    \filldraw[fill=white,draw=black,line width=0.7pt] (0,0) circle (0.62);
    \coordinate (p0) at (90:0.28);
    \coordinate (p1) at (18:0.28);
    \coordinate (p2) at (-54:0.28);
    \coordinate (p3) at (-126:0.28);
    \coordinate (p4) at (162:0.28);
    \draw[black!55,line width=0.25pt]
      (p0)--(p1) (p0)--(p2) (p0)--(p3) (p0)--(p4)
      (p1)--(p2) (p1)--(p3) (p1)--(p4)
      (p2)--(p3) (p2)--(p4) (p3)--(p4);
    \foreach \j in {0,...,4}
      \fill (p\j) circle (1.15pt);
    \node[font=\scriptsize,fill=white,inner sep=0.6pt] at (0,-0.42)
      {$B_{#3}$};
  \end{scope}%
}
\begin{tikzpicture}
  \coordinate (c) at (0,0);
  \coordinate (b0) at (0,3);
  \coordinate (b1) at (2.6,1.5);
  \coordinate (b2) at (2.6,-1.5);
  \coordinate (b3) at (0,-3);
  \coordinate (b4) at (-2.6,-1.5);
  \coordinate (b5) at (-2.6,1.5);

  \draw[black!60,line width=1.35pt]
    (b0)--(b1)--(b2)--(b3)--(b4)--(b5)--cycle;
  \foreach \j in {0,...,5}
    \draw[black!60,line width=1.35pt] (c)--(b\j);

  \wheelblock{0}{0}{*}
  \wheelblock{0}{3}{0}
  \wheelblock{2.6}{1.5}{1}
  \wheelblock{2.6}{-1.5}{2}
  \wheelblock{0}{-3}{3}
  \wheelblock{-2.6}{-1.5}{4}
  \wheelblock{-2.6}{1.5}{5}
\end{tikzpicture}
\endgroup
\caption{The block structure of \(W_k\), illustrated for \(k=6\). Each
circular block induces a copy of \(K_5\), and each heavy line represents the
complete join between the two corresponding blocks. Thus the graph obtained
by contracting every block is a wheel.}
\label{fig:wheel-perfect}
\end{figure}
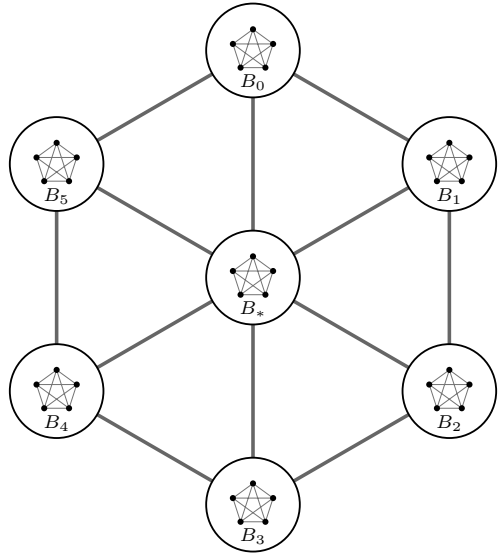

\begin{theorem}\label{tm:wheel-perfect}
Assume ASH\@. Then $W_k$ is perfectly pebblable for every integer $k\ge3$.
Moreover,
\begin{equation*}
|V(W_k)|=5(k+1)
\qquad\text{and}\qquad
|E(W_k)|=60k+10.
\end{equation*}
In particular,
\begin{equation*}
|E(W_k)|=O(k)=O\bigl(|V(W_k)|\bigr).
\end{equation*}
\end{theorem}

\begin{proof}
Fix $k\ge3$. We verify the conditions of Theorem~\ref{tm:ASH2PP}. The graph
$W_k$ is finite and connected because the central block $B_*$ is completely
joined to every rim block.

To verify condition~(1), let $x$ and $y$ be distinct vertices. If they belong
to the same block, any third vertex of that block is a common neighbor. If
they belong to distinct rim blocks, any vertex of $B_*$ is a common neighbor.
Finally, if one belongs to $B_*$ and the other to a rim block $B_i$, any
other vertex of $B_i$ is a common neighbor. Thus $W_k$ is
strongly edge-triangulated.

Let $X\subseteq V(W_k)$ with $|X|\le4$. Every block contains a vertex outside
$X$. List the vertices of each block outside $X$ in an arbitrary order, and
concatenate the lists in the order
\begin{equation*}
B_*,B_0,B_1,\dots,B_{k-1}.
\end{equation*}
Consecutive blocks in this list are completely joined, as are $B_{k-1}$ and
$B_*$. Hence the resulting cyclic list is a Hamiltonian cycle in $W_k-X$.
Breaking it immediately before any prescribed vertex gives a Hamiltonian path
starting at that vertex. This proves condition~(2).

Now let $X\subseteq V(W_k)$ with $|X|\ge5$, and let
\begin{equation*}
I=\{i\in\mathbb Z_k:B_i\subseteq X\}.
\end{equation*}
We also record whether $B_*\subseteq X$. Let $q$ be the total number of blocks
contained in $X$. Since these blocks are pairwise disjoint and have size five,
\begin{equation}\label{eq:wheel-deleted-blocks}
q\le\left\lfloor\frac{|X|}{5}\right\rfloor.
\end{equation}
If $q=0$, the preceding construction gives a Hamiltonian cycle in $W_k-X$, so
\begin{equation*}
\theta(W_k-X)=1\le\left\lfloor\frac{|X|}{5}\right\rfloor.
\end{equation*}

Suppose that $q\ge1$. If $I=\varnothing$, then $B_*\subseteq X$, and the
surviving rim blocks form a Hamiltonian cycle. Thus $\theta(W_k-X)=1\le q$.
Assume that $I\ne\varnothing$. The set $\mathbb Z_k\setminus I$ is the union
of at most $|I|$ maximal cyclic intervals of consecutive indices. Listing the
surviving vertices block by block along each interval gives at most $|I|$
pairwise vertex-disjoint paths covering all surviving rim vertices. If
$B_*\setminus X\ne\varnothing$ and at least one rim path is present, append
all vertices of $B_*\setminus X$, in an arbitrary order, to one of those
paths. If no rim vertex
survives but $B_*\setminus X\ne\varnothing$, the vertices of
$B_*\setminus X$ form one path by themselves. If no vertex survives, the
empty family is a path cover.
Consequently,
\begin{equation*}
\theta(W_k-X)\le q.
\end{equation*}
Together with~\eqref{eq:wheel-deleted-blocks}, this proves condition~(3).
Theorem~\ref{tm:ASH2PP} now implies that $W_k$ is perfectly pebblable.

There are $k+1$ blocks, each containing five vertices, and hence
\begin{equation*}
|V(W_k)|=5(k+1).
\end{equation*}
The edges within the blocks contribute
\begin{equation*}
(k+1)\binom{5}{2}=10(k+1)
\end{equation*}
edges. The underlying wheel has $k$ rim edges and $k$ edges from its center to
the rim. Each of these $2k$ edges corresponds to a complete join between two
blocks and therefore contributes $25$ edges. It follows that
\begin{equation*}
|E(W_k)|=10(k+1)+25(2k)=60k+10.
\end{equation*}
The asserted asymptotic estimate is immediate. More explicitly,
\begin{equation*}
\frac{|E(W_k)|}{\binom{|V(W_k)|}{2}}
=\frac{120k+20}{5(k+1)(5k+4)}
\sim\frac{24}{5k},
\end{equation*}
so this edge density also tends to zero.
\end{proof}

\subsection{Open problems}

These observations suggest several natural problems. In particular, one may
ask for a structural characterization of perfectly pebblable graphs, for the
complexity of recognizing them, and for the threshold function for perfect
pebblability in the random graph \(G(n,p)\).

\section*{Declaration of Generative AI and AI-assisted technologies in the
writing process}

During the preparation of this work, the authors used ChatGPT to correct
grammatical and typographical errors and misprints, to assist in checking
mathematical arguments, and to improve the presentation of the manuscript.
All mathematical results and constructions are due to the authors. After
using this tool, the authors reviewed and edited the content as needed and
take full responsibility for the content of the publication.

\end{document}